\documentclass[preprint,10pt]{article}
\usepackage{amsmath,amssymb,amsthm,latexsym}
\usepackage{graphicx,xcolor,tikz}

\renewenvironment{proof}{\noindent{\bf Proof.}}{~~$\Box$}
\theoremstyle{plain}
\newtheorem{thrm}{Theorem}[section]
\newtheorem{lemm}[thrm]{Lemma}
\newtheorem{prot}[thrm]{Proposition}

\newtheorem{corl}[thrm]{Corollary}

\theoremstyle{definition}
\newtheorem{dfnt}[thrm]{Definition}
\newtheorem{remk}[thrm]{Remark}
\newtheorem{exa}[thrm]{Example}

\begin{document}
\title{\bf A Radon-Nikodym theorem for monotone measures
}
\author{
Yao Ouyang\thanks{E-mail:oyy@zjhu.edu.cn(Y.
Ouyang)}\\
\small\it Faculty of Science, Huzhou Teacher's
College, Huzhou, Zhejiang 313000, \\ \small\it People's Republic of China \\
Jun Li\thanks{Corresponding author. E-mail:lijun@cuc.edu.cn(J. Li)}\\
\small\it State Key Laboratory of Media Convergence and Communication, \\ \small\it Communication University of China,
            Beijing 100024, China
}

\date{}
\maketitle
\begin{abstract}
\hspace{4mm}
A version of Radon-Nikodym theorem for the Choquet integral w.r.t. monotone measures is proved.
Without any presumptive condition, we obtain a necessary and sufficient condition for the ordered pair $(\mu, \nu)$ of finite monotone measures to have
the so-called Radon-Nikodym property related to a nonnegative measurable function $f$.
If $\nu$ is null-continuous and weakly null-additive, then $f$ is uniquely determined almost everywhere by $\nu$ and
thus is called the Radon-Nikodym derivative of $\mu$ w.r.t. $\nu$.
For $\sigma$-finite monotone measures, a Radon-Nikodym type theorem is also obtained under the assumption that
the monotone measures are lower continuous and null-additive.

\end{abstract}

\section{Introduction}
Suppose that $\nu$ is a $\sigma$-additive measure and $f$ is a nonnegative measurable function. The measure $\mu$ defined by
\[
\mu(A)=\int_A fd\nu
\]
for all measurable sets $A$ is said to be the indefinite integral of $f$ w.r.t. $\nu$. In this case, $\mu$ is absolutely continuous w.r.t. $\nu$, that is, $\mu(A)=0$ whenever $\nu(A)=0$.
Under what conditions a measure can be expressed as the indefinite integral w.r.t. another measure is quite interesting. This pertains to the scope of the Radon-Nikodym theorem.
The Radon-Nikodym theorem, one of the most important theorems in measure theory, states that for a $\sigma$-finite measure $\nu$, $\mu$ is the indefinite integral w.r.t. $\nu$ if and only if $\mu$ is absolutely continuous w.r.t. $\nu$, see Halmos \cite{Hal68} for example. We note that the Radon-Nikodym theorem has various proofs and all these proofs are highly dependent on the $\sigma$-additivity of measures.

When one of the measures is only finitely additive, the Radon-Nikodym theorem does not hold in general. Since in this case, the Hahn decomposition does not hold in general. Various conditions \cite{BelHag88,BerRegRig92,Boc39,CanMar92,May79} have been derived in the literature for the validity of finitely additive measures-based Radon-Nikodym theorem. Bochner \cite{Boc39} proved that if $\nu(A)\to 0$ implies $\mu(A)\to 0$ (it will be called that $\mu$ is strongly absolutely continuous w.r.t. $\nu$ in the following) then for any $\epsilon>0$ there is a simple function $f_\epsilon$ such that $|\int_A f_\epsilon d\nu-\mu(A)|<\epsilon$ for all measurable sets $A$. However, it was not until Maynard's work \cite{May79} that an exact Radon-Nikodym derivative was given. In \cite{BelHag88} the Radon-Nikodym theorem was proved under the strongly absolute continuity and a property called Hahn separation (a variant of Hahn decomposition), while \cite{BerRegRig92,CanMar92} proved the Radon-Nikodym theorem in terms of some properties of the range of $(\mu, \nu)$.

Graf \cite{Gra80} proved a Radon-Nikodym theorem for the Choquet integral w.r.t. capacities (lower continuous subadditive monotone measure), while Nguyen et al. \cite{Ngu06,NguNguWan97} investigated a Radon-Nikodym theorem for $\sigma$-subadditive monotone measures. Greco \cite{Gre81} (see also \cite{CanVol02}) obtained necessary and sufficient conditions of this theorem for null-additive monotone measures. Roughly speaking, these conditions include a variant of Hahn decomposition and some other conditions. We also note that R\'{e}bill\'{e} \cite{Reb13} discussed the superior Radon-Nikodym derivative for almost subadditive monotone measures.

In this paper, a new version of Radon-Nikodym theorem for the Choquet integral is proved. It should be stressed that our result generalizes the corresponding ones in \cite{Gra80,Gre81,Ngu06}.
Concretely, we introduce a decomposition property for monotone measures in Section 3.
This property concerns an ordered pair $(\mu, \nu)$ of monotone measures and a decreasing collection $\{A_\alpha\}_{\alpha\in\mathbb{Q}_+}$ of measurable sets
and is a natural generalization of the Hahn decomposition for $\sigma$-additive measures.
The decomposition property together with $\lim\limits_{\alpha\to\infty}\mu(A_\alpha)\vee\nu(A_\alpha)=0$ is demonstrated to be
the necessary and sufficient conditions for $(\mu, \nu)$ to have Radon-Nikodym property
 based on the Choquet integral (i.e., there is a nonnegative measurable function $f$ such that $\mu(E)=\int_E fd\nu$ for each measurable set $E$), where $\mu, \nu$ are finite monotone measures.
This result is obtained without any presumptive condition other than the monotonicity of set functions,
thus it is a generalization of the results of Greco \cite{Gre81} and Nguyen et al. \cite{NguNguWan97}.
When $\nu$ is further weakly null-additive (which is weaker than subadditive) and null-continuous (which is implied by lower continuous),
then the function $f$ is unique a.e.[$\nu$] and is called the Radon-Nikodym derivative of $\mu$ w.r.t. $\nu$. Thus, Graf's result is also generalized.
The Radon-Nikodym theorem for $\sigma$-finite monotone measures are considered in Section 4.
The existence and uniqueness of the Radon-Nikodym derivative is obtained when $\mu, \nu$ are lower continuous and $\nu$ is further null-additive.
\section{Preliminaries}\label{Sec-Preliminaries}
Let $(U, {\mathcal U})$ denote a measurable space, that is, a nonempty set $U$ equipped with a $\sigma$-algebra ${\mathcal U}$ of subsets of $U$. A subset $A$ of $U$ is called measurable (w.r.t. ${\mathcal U}$) if $A\in{\mathcal U}$. A nonnegative extended real-valued function $f\colon U\to \overline{\mathbb{R}}_{+}$ is called measurable if for each $\alpha\in [0, +\infty]$, $\{f\geq\alpha\}\in {\mathcal U}$ (here $\{f\geq\alpha\}$  is the abbreviation for $\{t\in U\, | \, f(t)\geq\alpha\}$).

\begin{dfnt}
A set function $\mu: {\mathcal U}\rightarrow \overline{\mathbb{R}}_{+}$ is called a {\it monotone measure} if it satisfies the following two conditions:

{\rm (i)} $\mu(\emptyset)=0$;            \hfill  (vanishing at $\emptyset$)

{\rm (ii)} $\mu(A) \leq \mu(B)$ whenever $A \subset B$ and $A, B \in {\mathcal{U}}$. \hfill  (monotonicity)

The triple $(U, {\mathcal U}, \mu)$ is called a \emph{monotone measure space}.

A monotone measure $\mu$ on $(U, {\mathcal U})$ is said to be (i) \emph{finite} if $\mu(U)<\infty$; (ii) \emph{$\sigma$-finite} if there is $\{U_n\}_{n=1}^\infty\subset{\mathcal U}$ with $U_n\nearrow U$ (\emph{i.e.}, $U_1\subset U_2\subset\cdots\subset U_n\subset\cdots$ and $\bigcup\limits_{n=1}^\infty U_n=U$) such that $\mu(U_n)<\infty$ for each $n$.
\end{dfnt}

Let $f$ be a nonnegative measurable function, 
$\nu$ be a monotone measure 
and $A$ be a measurable set. 
The Choquet integral of $f$ w.r.t. $\nu$ is defined as follows, see \cite{Ch53,Den94,OuyLiMes15,Pap95}.
\begin{dfnt}
The Choquet integral of $f$ w.r.t. $\nu$ on $A$ is given by
\[
\int_{A} fd\nu=\int_0^\infty \nu(\{f\geq\alpha\}\cap A)d\alpha,
\]
where the integral on the right side is the improper Riemann integral.
\end{dfnt}

When $A=U$, we write $\int fd\nu$ instead of $\int_{U} fd\nu$. If $\int fd\nu<\infty$, then $f$ is called Choquet integrable w.r.t. $\nu$ on $U$.
When $\nu$ is a $\sigma$-additive measure, the Choquet integral coincides with the Lebesgue
integral.

Throughout this paper, unless otherwise stated, all the considered integrals are assumed to be the Choquet integrals.
The following are some basic properties of the Choquet integrals (\cite{Den94,Pap95,WanKli09}\label{Prot-BasicPropofChoInt}):

\begin{prot}
Let $(U, {\mathcal U}, \nu)$ be a monotone measure space and $f, g$ be nonnegative measurable functions. Then

{\rm (i)} $\int_A fd\nu=0$ whenever $\nu(A)=0$;

{\rm(ii)} $f\leq g$ implies $\int fd\nu\leq\int gd\nu$; \hfill  \hfill  (monotonicity)


{\rm(iii)} $\int cfd\nu=c\int fd\nu$ for any constant $c\geq 0$; \hfill  (homogeneity)

{\rm (iv)} $\int_A\chi_A d\nu=\nu(A), \forall\, A\in{\mathcal U}$, where
 $\chi\sb{A}$ denotes the characteristic function of $A$;

{\rm (v)} $\int_A f d\nu=\int f\chi_A d\nu$;



{\rm (vi)} $\int_A f d\nu=\lim\limits_{n\to\infty}\int_A (f\wedge n)d\nu$.
\end{prot}


Two functions $f, g$ on $U$ are said to be \emph{comonotone} if for any $t_1, t_2\in U$, $(f(t_1)- f(t_2))(g(t_1)-
g(t_2))\geq 0$.
The following proposition is known as \emph{comonotonic additivity}  of the Choquet integral, which is a distinguishing feature of the Choquet integral, see \cite{Den94,Sch86}.

\begin{prot}\label{comonotone-ChoInt}
Let $(U, {\mathcal U}, \nu)$ be a monotone measure space and $f, g$ be nonnegative measurable functions. If $f$ and $g$ are comonotone, then
\[
\int(f+g)d\nu=\int fd\nu+\int gd\nu.
\]
\end{prot}

Note that a constant function $c$ is comonotone with arbitrary functions and, when $U$ is an ordered set then two increasing (decreasing, resp.) functions are comonotone. Moreover, for any function $f$ and any constant $c$, $(f-c)\vee 0$ and $f\wedge c$ are comonotone, where $(f\vee c)(t)=\max\{f(t), c\}$ and $(f\wedge c)(t)=\min\{f(t), c\}$.
\section{Radon-Nikodym theorem for finite monotone measures}
In this section we present a new version of Radon-Nikodym theorem for finite monotone measures.
To do this, we introduce the following concept of {\it decomposition property} relating to an ordered pair of monotone measures.
\begin{dfnt}\label{Dfnt-decompro}
Let $\mu, \nu$ be two monotone measures on $(U, {\mathcal U})$. The ordered pair $(\mu, \nu)$ is said to have \emph{decomposition property} if there is a decreasing collection $\{A_\alpha\}_{\alpha\in\mathbb{Q}^{+}}$ of measurable sets  with $A_0=U$ such that
\begin{eqnarray}\label{decompro-inequality-1}
\alpha\Big(\nu(A\cap A_\alpha)-\nu(A\cap A_\beta)\Big)
&\leq&  \mu(A\cap A_\alpha)-\mu(A\cap A_\beta) \\ \label{decompro-inequality-2}
&\leq&  \beta\Big(\nu(A\cap A_\alpha)-\nu(A\cap A_\beta)\Big)
\end{eqnarray}
holds for any $A\in{\mathcal U}$ with finite measures for $\mu$ and $\nu$, and
any $\alpha, \beta\in \mathbb{Q}^{+}$ with $\alpha<\beta$, where $\mathbb{Q}^{+}$ is the set of all nonnegative rational numbers.
\end{dfnt}

\begin{exa}
For any $\sigma$-additive finite measures $\mu, \nu$,
the ordered pair $(\mu, \nu)$ has the decomposition property w.r.t. $\{A_\alpha\}_{\alpha\in \mathbb{Q}^{+}}$, where
$(A_\alpha, A^{\rm c}_\alpha)$ is a Hahn decomposition of the signed measure $\mu-\alpha\nu$ (see Remark \ref{Nte-sigmaaddmeasure} for details).
This is why we call the ordered pair $(\mu, \nu)$ which satisfies inequalities (\ref{decompro-inequality-1}) and (\ref{decompro-inequality-2}) in Definition \ref{Dfnt-decompro} having the decomposition property.
\end{exa}

\begin{lemm}\label{Lemm-lemma4-2}
Let $(\mu, \nu)$ have the decomposition property w.r.t. $\{A_\alpha\}_{\alpha\in \mathbb{Q}^{+}}$. If $\lim\limits_{\alpha\to\infty}\mu(A_\alpha)\vee\nu(A_\alpha)=0$, then $\lim\limits_{\alpha\to\infty}\alpha\nu(A_\alpha)=0$.
\end{lemm}
\begin{proof}
We can assume that $\mu(A_\alpha)\vee\nu(A_\alpha)<\infty$.
From the first inequality in Definition~\ref{Dfnt-decompro}, we get
\[
\alpha\Big(\nu(A_\alpha)-\nu(A_\beta)\Big)\leq \mu(A_\alpha)-\mu(A_\beta)
\]
holds for any $\alpha<\beta$, and hence we have $\alpha\nu(A_\alpha)\leq\mu(A_\alpha)$ by letting $\beta\to\infty$
as $\lim\limits_{\beta\to\infty}\mu(A_\beta)\vee\nu(A_\beta)=0$. Thus we also have $\alpha\nu(A_\alpha)$ tends to $0$ whenever $\alpha$ tends to $\infty$.
\end{proof}


\bigskip
Now we show our main result --- a version of Radon-Nikodym theorem for finite monotone measures.

\begin{thrm}\label{R-N-thrm-finitecase}
Let $\mu, \nu$ be two finite monotone measures on $(U, {\mathcal U})$. Then the following two assertions are equivalent:

{\rm (i)} The ordered pair $(\mu, \nu)$ has \emph{Radon-Nikodym property}, i.e.,
 there is a nonnegative measurable function $f\colon U\to \overline{\mathbb{R}}_{+}$
 such that
\begin{equation}\label{Eq_R-N}
\mu(A)=\int_A fd\nu, \ \ \forall\, A\in{\mathcal U}.
\end{equation}

{\rm (ii)} The ordered pair $(\mu, \nu)$ has the decomposition property w.r.t. a sets system
$\{A_\alpha\}_{\alpha\in \mathbb{Q}^{+}}$ and
$\lim\limits_{\alpha\to\infty}\mu(A_\alpha)\vee \nu(A_\alpha)=0$.
\end{thrm}

\begin{proof}
(i)$\Longrightarrow$ (ii).
%
%
Suppose there is a nonnegative measurable function $f$
such that Eq. (\ref{Eq_R-N}) holds, then $f$ is Choquet integrable w.r.t. $\nu$ on $U$,
i.e., $\int fd\nu=\mu(U)<\infty$.

First we show that $(\mu, \nu)$ has the decomposition property. Put $A_\alpha=\{f\geq\alpha\}$, then $\{A_\alpha\}$ is decreasing and $A_0=U$.
For any $A\in{\mathcal U}$ with finite measure
(i.e., $\mu(A)\vee\nu(A)<\infty$), and any $\alpha<\beta$ it holds
\begin{eqnarray*}
\mu(A\cap A_\alpha)-\mu(A\cap A_\beta)&=&\int_{A\cap A_\alpha}fd\nu-\int_{A\cap A_\beta}fd\nu\\
&=&\int_0^\infty \Big(\nu(A\cap A_\alpha\cap A_t)-\nu(A\cap A_\beta\cap A_t)\Big)dt\\
&\geq&\int_0^\alpha \Big(\nu(A\cap A_\alpha)-\nu(A\cap A_\beta)\Big)dt\\
&=&\alpha\Big(\nu(A\cap A_\alpha)-\nu(A\cap A_\beta)\Big).
\end{eqnarray*}

On the other hand, we have
\begin{eqnarray*}
\mu(A\cap A_\alpha)-\mu(A\cap A_\beta)&=&\int_0^\beta \Big(\nu(A\cap A_\alpha\cap A_t)-\nu(A\cap A_\beta)\Big)dt\\
&&+ \int_\beta^\infty \Big(\nu(A\cap A_t)-\nu(A\cap A_t)\Big)dt\\
&\leq&\int_0^\beta \Big(\nu(A\cap A_\alpha)-\nu(A\cap A_\beta)\Big)dt\\
&=&\beta\Big(\nu(A\cap A_\alpha)-\nu(A\cap A_\beta)\Big).
\end{eqnarray*}

The assertion $\lim\limits_{\alpha\to\infty}\nu(A_\alpha)=0$ follows from $\alpha\nu(A_\alpha)\leq\int_{A_\alpha}f d\nu\leq\mu(U)<\infty$.

Since $f=((f-n)\vee 0)+ (f\wedge n)$ and $f=((f-n)\vee 0)$ and $(f\wedge n)$
are comonotone, then
\[
\int fd\nu=\int ((f-n)\vee 0)d\nu+\int (f\wedge n)d\nu
\]
holds for each $n$.
Therefore, from $\int fd\nu=\lim\limits_{n\to\infty}\int (f\wedge n)d\nu$
we conclude that
$$
\lim\limits_{n\to\infty}\int ((f-n)\vee 0)d\nu=0,
$$
as $\int fd\nu<\infty$.
 Also,
\begin{eqnarray*}
\mu(A)=\int_A fd\nu   &=& \int_A ((f-n)\vee 0)d\nu+\int_A (f\wedge n)d\nu \\
                    &\leq& \int ((f-n)\vee 0)d\nu+ n\nu(A)
\end{eqnarray*}
for each $A\in\mathcal{U}$. Specifically,
$$
\mu(A_n)\leq \int ((f-n)\vee 0)d\nu+ n\nu(A_n)
$$
for each $n$, it then follows that $\lim\limits_{n\to\infty}\mu(A_n)=0$.
Therefore, $\lim\limits_{\alpha\to\infty}\mu(A_\alpha)=0$ as
$\{A_\alpha\}_{\alpha\geq 0}$ is a decreasing collection of measurable sets.

 (ii)$\Longrightarrow$(i). Suppose that $(\mu, \nu)$ has the decomposition property and $\{A_\alpha\}_{\alpha\in \mathbb{Q}^{+}}$ is the corresponding sets system satisfying $\lim\limits_{\alpha\to\infty}\mu(A_\alpha)\vee\nu(A_\alpha)=0$.
We show that there is a nonnegative measurable function $f\colon U\to \overline{\mathbb{R}}_{+}$ such that Eq. (\ref{Eq_R-N}) holds.

Define $f\colon U\to [0, \infty]$ as
\begin{equation}\label{Eq-DecomFun}
f(x)=\sup\{\alpha\,|\,x\in A_\alpha\}.
\end{equation}
Then $f$ is measurable and $\mu(\{f=\infty\})=0$ as $\{f=\infty\}=\bigcap\limits_{\alpha\in\mathbb{N}}A_\alpha$.

For each positive integer $n$, define
\[
f_n(x)=\left\{
\begin{array}{ll}
    \frac{k-1}{2^n},             &\mbox{if\,\,} x\in A_{\frac{k-1}{2^n}}\setminus A_{\frac{k}{2^n}}, k=1, 2, \cdots, n2^n, \\
    n,   &\mbox{if\,\,} x\in A_n.
\end{array} \right.
\]
Then $\{f_n\}_{n\in\mathbb{N}}$ is an increasing sequence and $f\wedge n-\frac{1}{2^n}\leq f_n\leq f$.
Since $f_n$ can be rewritten as
\[
f_n=\frac{1}{2^n}\sum_{k=1}^{n\cdot 2^n}\chi_{A_\frac{k}{2^n}},
\]
for any given $A\in{\mathcal U}$,  we have
\begin{eqnarray*}
& & \int_A f_nd\nu =
\frac{1}{2^n}\sum_{k=1}^{n\cdot 2^n}\nu(A\cap A_\frac{k}{2^n})\\
&=&
\sum_{k=1}^{n\cdot 2^n-1}\frac{k}{2^n}\Big(\nu(A\cap A_\frac{k}{2^n})-\nu(A\cap A_\frac{k+1}{2^n})\Big)+n\nu(A\cap A_n)\\
&\leq& \sum_{k=1}^{n\cdot 2^n-1}\Big(\mu(A\cap A_\frac{k}{2^n})-\mu(A\cap A_\frac{k+1}{2^n})\Big)+n\nu(A\cap A_n)\\
&=& \mu(A\cap A_\frac{1}{2^n})-\mu(A\cap A_n)+ n\nu(A\cap A_n).
\end{eqnarray*}
It follows from the assumption and
Lemma \ref{Lemm-lemma4-2} that
\[
\lim\limits_{n\to\infty}\int_A f_nd\nu\leq \lim\limits_{n\to\infty}\mu(A\cap A_\frac{1}{2^n})\leq\mu(A).
\]
The inequality $f\wedge n\leq f_n+ \frac{1}{2^n}$ implies that
\[
\int_A (f\wedge n)d\nu\leq \int_A f_nd\nu+ \int_A \frac{1}{2^n}d\nu= \int_A f_nd\nu+ \frac{1}{2^n}\nu(A).
\]
By virtue of Proposition \ref{Prot-BasicPropofChoInt}(vi) we get
\[
\int_A fd\nu=\lim\limits_{n\to\infty}\int_A (f\wedge n)d\nu\leq\lim\limits_{n\to\infty}\int_A f_nd\nu\leq\mu(A).
\]

On the other hand,
\begin{eqnarray*}
& & \int_A f_nd\nu =
\frac{1}{2^n}\sum_{k=1}^{n\cdot 2^n}\nu(A\cap A_\frac{k}{2^n})\\
&=&
\sum_{k=1}^{n\cdot 2^n-1}\frac{k+1}{2^n}\Big(\nu(A\cap A_\frac{k}{2^n})-\nu(A\cap A_\frac{k+1}{2^n})\Big) \\
& & \ \ \ \ \ \ \ \ \
+ \ (n+\frac{1}{2^n})\nu(A\cap A_n)-\frac{1}{2^n}\nu(A\cap A_{\frac{1}{2^n}})\\
&\geq&
\sum_{k=1}^{n\cdot 2^n-1}\Big(\mu(A\cap A_\frac{k}{2^n})-\mu(A\cap A_\frac{k+1}{2^n})\Big)\\
& & \ \ \ \ \ \ \ \ \
+ \ (n+\frac{1}{2^n})\nu(A\cap A_n)-\frac{1}{2^n}\nu(A\cap A_{\frac{1}{2^n}})\\
&\geq& \mu(A\cap A_\frac{1}{2^n})-\mu(A\cap A_n)+(n+\frac{1}{2^n})\nu(A\cap A_n)-\frac{1}{2^n}\nu(A\cap A_{\frac{1}{2^n}}).
\end{eqnarray*}
By the decomposition property we have
\begin{eqnarray*}
\mu(A)-\mu(A\cap A_\frac{1}{2^n})
&=&\mu(A\cap A_0)-\mu(A\cap A_\frac{1}{2^n})\\
&\leq& \frac{1}{2^n}(\nu(A\cap A_0)-\nu(A\cap A_\frac{1}{2^n}))\to 0\, (n\to\infty),
\end{eqnarray*}
\emph{i.e.}, $\mu(A\cap A_\frac{1}{2^n})\to\mu(A)\, (n\to\infty)$. Since both $\mu(A_n)$ and $n\nu(A_n)$ tend to $0$ when $n\to\infty$,
it then holds that
\[
\int_A fd\nu\geq\lim\limits_{n\to\infty}\int_A f_nd\nu\geq\mu(A)
\]
as $f\geq f_n$ for each $n$. Thus we reach Eq. (\ref{Eq_R-N}), i.e.,
\begin{equation*}
\mu(A)=\int_A fd\nu, \ \ \forall\, A\in{\mathcal U}.
\end{equation*}

The proof is complete.
\end{proof}

The Radon-Nikodym theorem for $\sigma$-additive measures concerns the absolute continuity of measures \cite{Hal68}.
For monotone measures, there are various types of absolute continuity
(see \cite{LMZ2010,OuyLi04,Wan96,Wang1996a}). Let $\mu, \nu$ be two monotone measures on $(U, {\mathcal U})$.
(1) If for any $A\in{\mathcal U}$, $\nu(A)=0$ implies $\mu(A)=0$,
then we say that $\mu$ is absolutely continuous w.r.t. $\nu$
and denoted by $\mu\ll\nu$.
(2) If for each $\epsilon>0$ there is a $\delta>0$ such that $\mu(A)<\epsilon$
for all sets $A\in{\mathcal U}$ satisfying \textcolor{red}{$\nu(A)<\delta$}, then we say that $\mu$ is strongly absolutely continuous w.r.t. $\nu$ and denoted by $\mu\ll^{s}\nu$ (\cite{LMZ2010}).

Obviously,  $\mu\ll^{s}\nu$ implies $\mu\ll\nu$, but the converse is not true.

Observe that Theorem~\ref{R-N-thrm-finitecase}(ii) implies $\mu\ll\nu$
and $\mu\ll^{s}\nu$.
In fact, assume $\nu(A)=0$. From the second inequality in Definition~\ref{Dfnt-decompro}, we take $\alpha=0, \beta>0$, then
$$
\mu(A)- \mu(A\cap A_{\beta})\leq \beta\Big(\nu(A)- \nu(A\cap A_{\beta})\Big),
$$
which implies $\mu(A)- \mu(A\cap A_{\beta})=0$ for any $\beta>0$. Therefore,
$\mu(A)=0$ as $\lim_{\beta\to\infty}\mu(A_\beta)=0$.
Similarly, $\mu\ll^{s}\nu$ is also true.

Thus, we obtain necessary conditions that the Radon-Nikodym theorem in classical measure theory remains valid for the Choquet integral w.r.t. monotone measures (see also \cite{Wang1996a}).
\begin{corl}\label{R-N-AC}
Let $\mu, \nu$ be two finite monotone measures on $(U, {\mathcal U})$.
If there is a nonnegative measurable function $f\colon U\to \overline{\mathbb{R}}_{+}$
 such that Eq. (\ref{Eq_R-N}) holds, i.e.,
\begin{equation*}
\mu(A)=\int_A fd\nu, \ \ \forall\, A\in{\mathcal U},
\end{equation*}
then $\mu\ll^{s}\nu$.
\end{corl}

Note that the measurable function $f$ in Theorem \ref{R-N-thrm-finitecase} is not unique in general.

\begin{exa}\label{Exa-withoutWeaklyNull-add.}
Let $U$ be the set of all positive integers, $\mathcal U$ the power set of $U$ and
\[
A_\alpha=\left \{
        \begin {array}{ll}
         U,
                  &\quad \text{if}\ \alpha\in [0, 1]\cap\mathbb{Q},  \\[1mm]
         \{2, 4, 6, \cdots\},
                  &\quad \text{if}\ \alpha\in (1, 2]\cap\mathbb{Q}, \\[1mm]
         \emptyset,
                  &\quad \text{if}\ \alpha\in (2, \infty)\cap\mathbb{Q}.
        \end {array}
       \right.
\]
Define
\[
\mu(A)=\nu(A)=\left \{
        \begin {array}{ll}
         1,
                  &\quad \text{if\,} A=U,  \\[1mm]
         0,
                  &\quad \text{otherwise.}
        \end {array}
       \right.
\]
It is routine to verify that $(\mu, \nu)$ has the decomposition property w.r.t. $\{A_\alpha\}_{\alpha\in \mathbb{Q}^{+}}$. The condition $\lim\limits_{\alpha\to\infty}\mu(A_\alpha)\vee\nu(A_\alpha)=0$ is obviously satisfied.
By Theorem \ref{R-N-thrm-finitecase} there exists a nonnegative function $f$ such that $\mu(A)=\int_A fd\nu$ for each $A$. In fact
\[
f_1(x)=\sup\{\alpha \ | \ x\in A_\alpha\}=\left \{
        \begin {array}{ll}
         1,
                  &\quad \text{if\, $x$ is odd} \\[1mm]
         2,
                  &\quad \text{if\, $x$ is even}
        \end {array}
       \right.
\]
is such a function. Note that $f_2$ defined by
\[
f_2(x)=\left \{
        \begin {array}{ll}
         1,
                  &\quad \text{if\, $x$ is even} \\[1mm]
         2,
                  &\quad \text{if\, $x$ is odd}
        \end {array}
       \right.
\]
also satisfies $\mu(A)=\int_A f_2d\nu$ (there are in fact infinitely many such functions).
Interestingly, $f_1$ and $f_2$ are different at every point and thus $\nu(\{f_1\neq f_2\})=\nu(U)=1$.
\end{exa}

To ensure the uniqueness of $f$, we have to impose some additional conditions. Recall that a monotone measure $\mu$ is said to be
(i) \ {\it weakly null-additive} \cite{WanKli09}, if $\mu(A_{1}\cup A_{2})=0$ for any $A_{1}, A_{2}\in{\mathcal U}$ with $\mu(A_{1})=\mu(A_{2})=0$; (ii) \ {\it null-continuous} \cite{AsaMuro2006},
 if  $\mu(\bigcup_{n=1}^{\infty}A_n) = 0$
for every increasing sequence $\{A_n\}_{n\in N} \subset
{\mathcal{U}}$ such that $\mu(A_n) = 0, n=1,2,\cdots.$
The monotone measure $\mu$ is both weakly null-additive and null-continuous if and only if
$\mu(\bigcup_{n=1}^{\infty}A_n) = 0$ whenever $\{A_n\}_{n\in \mathbb{N}} \subset {\mathcal{U}}$ and $\mu(A_n) = 0, n=1,2,\cdots$, see \cite{Li2000}.
Such a monotone measure $\mu$ is called to have {\it property} ($\sigma$), i.e.,
the set of all $\mu$-null sets is a $\sigma$-ideal, see \cite{CanVol02}.

\begin{prot}\label{Prot-uniqueness}
Let $\mu, \nu$ be two monotone measures and $\nu$ is weakly null-additive and null-continuous.
If measurable functions $f, g\colon U\to [0, \infty]$ satisfy Eq. (\ref{Eq_R-N}), i.e.,
\[
\mu(A)=\int_A fd\nu=\int_A gd\nu, \ \ \forall\,A\in{\mathcal U},
\]
then $f=g \ a.e.[\nu]$ (i.e., $\nu(\{f\neq g\})=0$).

\end{prot}

\begin{proof}
Let $A=\{f>g\}$ and $A_n=\{f>g+\frac{1}{n}\}$, then $A_n\nearrow A$. We conclude that $\nu(A)= 0$, otherwise there is some $n$ such that $\nu(A_n)>0$ as $\nu$ is null-continuous. Then
\begin{eqnarray*}
\int_{A_n}fd\nu\geq\int_{A_n}(g+\frac{1}{n})d\nu
&=& \int_{A_n}gd\nu+\int_{A_n}\frac{1}{n}d\nu\\
&=& \int_{A_n}gd\nu+\frac{1}{n}\nu(A_n)>\int_{A_n}gd\nu,
\end{eqnarray*}
a contradiction. It holds similarly that $\nu(B)=0$, where $B=\{g>f\}$. Since $\nu$ is weakly null-additive, we have $\nu(A\cup B)=0$. As a consequence
$\nu(\{f\neq g\}) = \nu(A\cup B)=0.$
\end{proof}

Note: We also obtain $f=g \ a.e.[\mu]$ (i.e., $\mu(\{f\neq g\})=0$).

Now by using Proposition~\ref{Prot-uniqueness} we can propose the concept
of Radon-Nikodym derivative for monotone measures.
In Theorem~\ref{R-N-thrm-finitecase}, we consider the case that $\nu$ is
weakly null-additive and null-continuous (i.e., $\nu$ has {\it property} ($\sigma$)),
then the measurable function $f$ on $U$ for which Eq. (\ref{Eq_R-N}) holds
is called a Radon-Nikodym derivative (or Radon-Nikodym density) of $\mu$ w.r.t. $\nu$, and denoted by $\frac{d\mu}{d\nu}$ and Eq. (\ref{Eq_R-N}) will be written as $f=\frac{d\mu}{d\nu}$ or $d\mu =fd\nu$.
Thus, 
if the ordered pair $(\mu, \nu)$ satisfies the condition (ii) in Theorem~\ref{R-N-thrm-finitecase} and $\nu$ has {\it property} ($\sigma$), then any two Radon-Nikodym derivatives of $\mu$ w.r.t. $\nu$ are equal $a.e. \ [\nu]$,
and so the notation $\frac{d\mu}{d\nu}$ is only ambiguous up to a $\nu$-null set.

We can discuss some properties of Radon-Nikodym derivative. For example,
suppose that $\mu, \lambda$ and $\nu$ are finite monotone measures on $(U, {\mathcal U})$
and $\nu$ has {\it property} ($\sigma$), and both $(\mu,\nu)$ and $(\lambda,\nu)$
satisfy the condition (ii) in Theorem~\ref{R-N-thrm-finitecase}, respectively,
then $\frac{d\mu}{d\nu}$ and $\frac{d\lambda}{d\nu}$ exist.
We write $\frac{d\mu}{d\nu}=f$ and $\frac{d\lambda}{d\nu}=g$,
if $f$ and $g$ are comonotone, then we have
\[
\frac{d(\mu+\lambda)}{d\nu}= \frac{d\mu}{d\nu} + \frac{d\lambda}{d\nu} \ \  a.e.[\nu].
\]

\begin{remk}\label{Nte-sigmaaddmeasure}
Let $\mu, \nu$ be $\sigma$-additive finite measures. For each nonnegative rational number $\tau$ the signed measure $\mu-\tau\nu$ has a Hahn decomposition, \emph{i.e.},
there is a measurable set $A_\tau$ such that $A_\tau$ is a positive set of $\mu-\tau\nu$ and $A^{\rm c}_\tau$ is a negative set of $\mu-\tau\nu$.
Since $A_\tau$ is also a positive set of $\mu-\gamma\nu$ for any $\gamma\leq\tau$,
without loss of generality we can suppose that $\{A_\tau\}_{\tau\in \mathbb{Q}^{+}}$ is decreasing.
If $\tau=0$, then $U$ itself is a positive set of $\mu-\tau\nu=\mu$ and so $A_0=U$.
Let $\alpha<\beta$ be given. Since $A_\alpha$ is a positive set of $\mu-\alpha\nu$,
for any $A\in{\mathcal U}$, $(\mu-\alpha\nu)(A\cap (A_\alpha\setminus A_\beta))\geq 0$, \emph{i.e.},
\[
\alpha\Big(\nu(A\cap A_\alpha)-\nu(A\cap A_\beta)\Big)\leq \mu(A\cap A_\alpha)-\mu(A\cap A_\beta).
\]
Similarly, $A^{\rm c}_\beta$ is a negative set of $\mu-\beta\nu$ and hence is a positive set of $\beta\nu-\mu$.
For any $A\in{\mathcal U}$, $(\beta\nu-\mu)(A\cap (A_\alpha\setminus A_\beta))\geq 0$, \emph{i.e.},
\[
\mu(A\cap A_\alpha)-\mu(A\cap A_\beta)\leq\beta\Big(\nu(A\cap A_\alpha)-\nu(A\cap A_\beta)\Big).
\]
Thus $(\mu, \nu)$ has the decomposition property and $(A_\alpha, A^{\rm c}_\alpha)$ is a Hahn decomposition of the signed measure $\mu-\alpha\nu$.

If $\mu, \nu$ further satisfy $\mu\ll\nu$, then we have $\lim\limits_{\alpha\to\infty}\nu(A_\alpha)=0$ and hence $\lim\limits_{\alpha\to\infty}\mu(A_\alpha)=0$.
In fact, if $\lim\limits_{\alpha\to\infty}\nu(A_\alpha)>0$, then $(\mu-\alpha\nu)(A_\alpha)\to-\infty$ when $\alpha\to\infty$,
contradicting with the fact that $A_\alpha$ is a positive set of $\mu-\alpha\nu$. On the other hand, if $\lim\limits_{\alpha\to\infty}\mu(A_\alpha)\vee\nu(A_\alpha)=0$,
then $\mu\ll\nu$. To see this, let $\nu(A)=0$ be given. For any $\beta>0$,
\[
\mu(A\cap A_0)-\mu(A\cap A_\beta)\leq\beta(\nu(A\cap A_0)-\nu(A\cap A_\beta))=0,
\]
which implies $\mu(A)=0$ as $A_0=U$ and $\mu(A_\beta)\to 0\, (\beta\to\infty)$.

In conclusion, for two $\sigma$-additive finite measures $\mu, \nu$, the pair $(\mu,\nu)$ has the decomposition property, and $\mu\ll\nu$ if and only if
$\lim\limits_{\alpha\to\infty}\mu(A_\alpha)\vee\nu(A_\alpha)=0$.
\end{remk}

As a special case of Theorem~\ref{R-N-thrm-finitecase}, we thus obtain the classical Radon-Nikodym theorem.

\begin{corl}
Let $\mu, \nu$ be $\sigma$-additive finite measures on $(U, {\mathcal U})$. There exists a nonnegative measurable function $f\colon U\to\overline{\mathbb{R}}_{+}$ such that
\[
\mu(A)=\int_A fd\nu, \ \ \forall\, A\in{\mathcal U}
\]
if and only if $\mu\ll\nu$. 
\end{corl}

Let $\mu, \nu$ be bounded finitely additive measures such that for every $\epsilon>0$ there exists a finite decomposition of $U$,
$\{A_1,\cdots, A_n\}\subset{\mathcal U}$, satisfying $\mu(A_i)\vee\nu(A_i)<\epsilon$. For such measures,  Candeloro and Martellotti
proved in \cite{CanMar92} that if $\mu\ll^s\nu$ and the set $\{(\mu(A), \nu(A))| A\in{\mathcal U}\}$ is closed, then $(\mu, \nu)$ satisfies all requirements in Theorem \ref{R-N-thrm-finitecase}(ii) and thus $(\mu, \nu)$ has Radon-Nikodym property.


\begin{remk}
There are several papers dealing with Radon-Nikodym theorem for monotone measures (see,
for example, Graf \cite{Gra80}, Greco \cite{Gre81} and Nguyen et al. \cite{Ngu06,NguNguWan97}) and for set valued submeasures \cite{CanCroGav18,DiMusSam22}.
Graf obtained his result under the assumption that the monotone measures are subadditive and lower continuous, while Nguyen et al. demanded
the monotone measures being $\sigma$-subadditive. Greco \cite{Gre81} (see also Theorem 1.2 in Candeloro, Vol\v{c}i\v{c} \cite{CanVol02}) posed an additional requirement that
\[
(*) \ \ \ \ \  \mu(S)=\nu(S)=0\Rightarrow \nu(A\cup S)=\nu(A),\, \
\forall A\in{\mathcal U}.
\]
Specifically, under the condition $(*)$,
Greco proved that there is a nonnegative function $f$ such that
$$
\mu(A)=\int_A fd\nu,\, \ \forall A\in{\mathcal U}
$$
if and only if $(\mu, \nu)$ satisfies a strong decomposition property (S.D.P. for short, see \cite{Gra80}) w.r.t. a sets system $\{A_\alpha\}$ and $\lim\limits_{\alpha\to\infty}\mu(A_\alpha)=0$.
It is not difficult to see that S.D.P. together with $\lim\limits_{\alpha\to\infty}\mu(A_\alpha)=0$ implies
that $\mu\ll\nu$, hence the condition $(*)$ says in fact that $\nu$ is null-additive.

In contrast to these results, Theorem \ref{R-N-thrm-finitecase} has no additional requirements for monotone measures other than a set of sufficient and necessary conditions, thus it is a generalization of the results of Greco \cite{Gre81} and Nguyen et al. \cite{NguNguWan97}. Interestingly, Example \ref{Exa-withoutWeaklyNull-add.} shows that the Radon-Nikodym theorem can hold even for monotone measures without weakly null-additivity. When $\nu$ is further weakly null-additive and null-continuous then the Radon-Nikodym derivative is proved to be unique (up to a null set), see Proposition \ref{Prot-uniqueness}. Note subadditive and lower continous imply weakly null-additive and null-continuous (but not vice versa), thus the results of Graf are also generalized.
\end{remk}

\section{Radon-Nikodym theorem for $\sigma$-finite monotone measures}
Before presenting a Radon-Nikodym theorem for $\sigma$-finite monotone measures, we need some further properties of the Choquet integral.

Note that if $(\mu, \nu)$ has the decomposition property w.r.t. a sets system $\{A_\alpha\}_{\alpha\in \mathbb{Q}^{+}}$, then for any nonempty set $V\in{\mathcal U}$ the ordered pair $(\mu|_V, \nu|_V)$ also has \textcolor{red}{the} decomposition property w.r.t. the system $\{A_\alpha\cap V\}_{\alpha\in \mathbb{Q}^{+}}$.

A monotone measure $\mu$ is said to be (i) \emph{lower continuous}
(or \emph{continuous from below}) if for any $\{A_n\}\subset{\mathcal U}$
with $A_n\nearrow A$, it holds that $\mu(A)=\lim\limits_{n\to\infty}\mu(A_n)$;
(ii) \emph{null-additive} if $\mu(A\cup N)=\mu(A)$ for any $A, N\in{\mathcal U}$ with $\mu(N)=0$.

Obviously, lower continuity implies null-continuity and null-additivity
implies weak null-continuity, but not vice versa (see \cite{Li2000}).

\begin{prot}\cite{Den94,SonLi05}\label{Prot-furtherpropofCho}
Let $(U, {\mathcal U}, \nu)$ be a monotone measure space and $f, g, f_n \ (n=1, 2,\cdots)$ be nonnegative measurable functions.

{\rm (i)} If $f=g \ a.e.[\nu]$  and $\nu$ is null-additive, then $\int fd\nu=\int gd\nu$.

{\rm (ii)} If $\nu$ is lower continuous and $f_n\nearrow f$, then $\lim\limits_{n\to\infty}\int f_nd\nu=\int fd\nu$.
\end{prot}

In the following we suppose that the monotone measures $\mu, \nu$ are $\sigma$-finite.
Without loss of generality, we can assume that
there is $\{U_n\}_{n=1}^\infty\subset{\mathcal U}$ with $U_n\nearrow U$ such that for every $n$, $\mu(U_n)<\infty$ and $\nu(U_n)<\infty$ hold simultaneously.

\begin{thrm}\label{thrm-sigmafinite}
Let $\mu, \nu$ be $\sigma$-finite and lower continuous and $\nu$ be null-additive.
If $(\mu, \nu)$ has the decomposition property w.r.t. the system $\{A_\alpha\}_{\alpha\in \mathbb{Q}^{+}}$ and
$\lim\limits_{\alpha\to\infty}\mu(A_\alpha\cap U_n)\vee\nu(A_\alpha\cap U_n)=0$ for each $n$,
then there is a nonnegative and finite a.e.[$\nu$] measurable function $f$ such that
$$
\mu(A)=\int_A fd\nu, \ \ \forall\, A\in{\mathcal U}.
$$
In this case, $f$ is unique $a.e.[\nu]$.
\end{thrm}

\begin{proof}
For each $n$, $(\mu|_{U_n}, \nu|_{U_n})$ also has the decomposition property w.r.t. the system $\{A_\alpha\cap U_n\}_{\alpha\in \mathbb{Q}^{+}}$.
Since $\lim\limits_{\alpha\to\infty}\mu(A_\alpha\cap U_n)\vee\nu(A_\alpha\cap U_n)=0$ also holds, according to Theorem \ref{R-N-thrm-finitecase},
there is a nonnegative measurable function $f_n$ on $U_n$ such that
\[
\mu|_{U_n}(E)=\int_E f_nd\nu|_{U_n}
\]
for each measurable subset $E$ of $U_n$. Equivalently, for each $A\in{\mathcal U}$ we have
\[
\mu|_{U_n}(A_n)=\int_{A_n} f_nd\nu|_{U_n},
\]
where $A_n=A\cap U_n$. By Proposition~\ref{Prot-uniqueness}, for $n>m$ we have
$f_n|_{U_m}=f_m \ a.e.[\nu]$.
Without loss of generality, by Proposition \ref{Prot-furtherpropofCho}(i) we can assume that $f_n|_{U_m}=f_m$ as $\nu$ is null-additive.
Let $\tilde{f}_n(u)=f_n(u)$ for $u\in U_n$ and $\tilde{f}_n(u)=0$ for $u\in U\setminus U_n$. Then
\[
\mu(A_n)=\int_{A_n} \tilde{f}_nd\nu
\]
holds for each $n$. Note that the sequence $\{\tilde{f}_n\}$ is nondecreasing and thus it is convergent everywhere. Denote $f=\lim\limits_{n\to\infty} \tilde{f}_n$, then $f|_{U_n}=\tilde{f}_n$ for each $n$. Thus
\[
\{f=\infty\}=\bigcup_{n=1}^\infty\Big(U_n\cap\{f=\infty\}\Big)=\bigcup_{n=1}^\infty\{\tilde{f}_n=\infty\},
\]
which implies that $f$ is finite a.e.[$\nu$] as $\nu$ is null-additive and lower continuous. Moreover, $f|_{U_n}=\tilde{f}_n$ also implies that
\[
\mu(A_n)=\int_{A_n} fd\nu.
\]
Since $A_n\nearrow A$, again by the lower continuity of $\mu, \nu$ we reach the final conclusion
\[
\mu(A)=\lim_{n\to\infty}\mu(A_n)=\lim_{n\to\infty}\int_{A_n} fd\nu=\int_{A} fd\nu.
\]

The uniqueness of $f$ follows from (i) of Proposition \ref{Prot-furtherpropofCho}.

This completes the proof.
\end{proof}

\begin{exa}
 Let $U$ be the set of natural numbers and ${\mathcal U}$ be the power set of $U$. Let $\nu(A)=1$ for $A\neq\emptyset$, $\mu(A)=\max A$ if $A$ is finite
 and $\mu(A)=\infty$ if $A$ is infinite. Then (i) $\nu$ is lower continuous and null-additive; (ii)$\mu$ is lower continuous;
 (iii) $\nu$ is finite and $\mu$ is $\sigma$-finite.
 Let $A_\alpha=[\alpha, \infty)\cap U$ for each $\alpha\in \mathbb{Q}^{+}$.
 Then we can verify that $(\mu, \nu)$ has the decomposition property w.r.t. this system
 $\{A_\alpha\}_{\alpha\in \mathbb{Q}^{+}}$.
 It is easy to see that $\mu(A)=\int_A fd\nu, \forall\, A\in{\mathcal U}$ for $f(x)=\sup\{\alpha|x\in A_\alpha\}=x$.
 Note that $\lim\limits_{\alpha\to\infty}\mu(A_\alpha)=\infty$ and $\lim\limits_{\alpha\to\infty}\nu(A_\alpha)=1$. But $\lim\limits_{\alpha\to\infty}\mu(A_\alpha\cap U_n)\vee \nu(A_\alpha\cap U_n)=0$ for each $n$, where $U_n=\{0, 1, 2, \cdots, n\}$.
\end{exa}

\section*{Acknowledgements}
This work was supported by the Fundamental Research Funds for the Central Universities (China).

The authors would like to thank two reviewers for their constructive comments, which greatly improved the paper.

\end{document}